\documentclass[12pt]{article}

\usepackage{amssymb,amsthm,amsmath}
\usepackage[utf8]{inputenc}

\newcommand{\1}{\textbf{1}}

\newcommand{\call}[4]{\int_{#1}^{#2} {#3} \; \textrm{d} {#4}}

\newcommand{\mb}[1]{\mathbb{#1}}

\newcommand{\norma}[2]{\left\Vert #1 \right\Vert_{#2}}
\newcommand{\e}{\varepsilon}

\newcommand{\cube}{\{-\gamma,\gamma^{-1}\}^n}
\newcommand{\skal}[2]{\left\langle #1, #2 \right\rangle}
\newcommand{\kostka}{\{-1,1 \}^n }
\newcommand{\p}[1]{\mb{P}\left( #1  \right)}
\newcommand{\aff}{\mathcal{A}}
\newcommand{\affb}{\mathcal{A}_{[-1,1]}}

\DeclareMathOperator{\sgn}{\textrm{sgn}}
\DeclareMathOperator{\dist}{dist}

\newtheorem{thm}{Theorem}
\newtheorem{lem}{Lemma}

\theoremstyle{definition}
\newtheorem*{def*}{Definition}
\newtheorem*{question*}{Question}

\theoremstyle{remark}
\newtheorem*{rem*}{Remark}

\title{FKN Theorem on the biased cube}

\author{Piotr Nayar \thanks{Research partially supported by NCN Grant no. 2011/01/N/ST1/01839.}  }
\date{}

\begin{document}

\maketitle

\begin{abstract}
In this note we consider Boolean functions defined on the discrete cube $\{-\gamma,\gamma^{-1}\}^n$ equipped with a product probability measure $\mu^{\otimes n}$, where $\mu=\beta \delta_{-\gamma }+\alpha \delta_{ \gamma^{-1} }$ and $\gamma=\sqrt{\alpha/ \beta}$. We prove that if the spectrum of such a function is concentrated on the first two Fourier levels, then the function is close to a certain function of one variable. 

Moreover, in the symmetric case $\alpha=\beta=\frac12$ we prove that if a $[-1,1]$-valued function defined on the discrete cube is close to a certain affine function, then it is also close to a $[-1,1]$-valued affine function.
\end{abstract} 

\noindent {\bf 2010 Mathematics Subject Classification.} Primary 42C10; Secondary 60E15.

\noindent {\bf Key words and phrases.} Boolean functions, Walsh-Fourier expansion, FKN Theorem

\section{Introduction and notation}\label{sec:intro}

Let $\alpha,\beta >0$ with $\alpha+\beta=1$ and $\alpha \in (0,\frac12)$. We consider the discrete cube $\{-\gamma,\gamma^{-1} \}^n$ equipped with the $L_2$ structure given by the product probability measure $\mu_n = \mu^{\otimes n}$, where $\mu=\beta \delta_{-\gamma }+\alpha \delta_{ \gamma^{-1} }$ and $\gamma=\sqrt{\alpha/ \beta}$. For $f,g:\cube \to \mb{R}$ the standard scalar product $\skal{f}{g} = \call{}{}{fg}{\mu_n}$ induces the norm $\norma{f}{}=\sqrt{\skal{f}{f}}$. We also define the $L_p$ norm, $\norma{f}{p}=\left(\call{}{}{|f|^p}{\mu_n}\right)^{1/p}$. Let $[n]=\{1,2\ldots,n\}$. For $T \subseteq [n]$ and $x=(x_1,\ldots,x_n)$ let $w_T(x) = \prod_{i \in T} x_i$ and $w_\emptyset \equiv 1$. Note that we have $\call{}{}{x_i}{\mu_n}=0$ and $\call{}{}{x_ix_j}{\mu_n}= \delta_{ij}$. It follows that $(w_T)_{T \subseteq [n]}$ is an orthonormal basis of $L_2(\cube,\mu_n)$. Therefore, every function $f:\cube \to \mb{R}$ admits the unique expansion $f=\sum_{T \subset [n]} a_T w_T$. The functions $w_T$ are sometimes called the Walsh-Fourier functions. If the function $f$ is $\{-1,1 \}$-valued then it is called Boolean.

The Fourier analysis of Boolean functions plays an important role in many areas of research, including learning theory, social choice, complexity theory and random graphs, see e.g. \cite{O1}  and \cite{O2}. One of the most important analytic tools in this theory is the so-called hypercontractive Bonami-Beckner-Gross inequality, see \cite{Bo}, \cite{Be}, \cite{G1} and \cite{G2} for a survey on this topic. This inequality has been used in the celebrated papers by J. Kahn, G. Kalai and N. Linial, \cite{KKL}, and E. Friedgut, \cite{F}. It can be stated as follows. Take $\alpha=\beta=\frac12$ and $q \in [1,2]$. Then we have      
\begin{equation}\label{hipsym}
	\norma{\sum_{T \subseteq [n]} (q-1)^{|T|/2}  a_T w_T  }{2} \leq \norma{\sum_{T \subseteq [n]}  a_T w_T  }{q}
\end{equation}
for every choice of $a_T \in \mb{R}$. This inequality has been generalized in \cite{Ol} to the non-symmetric case. Namely, the following inequality holds true,
\begin{equation}\label{hipbiased}
	\norma{\sum_{T \subseteq [n]} c_q(\alpha,\beta)^{|T|}  a_T w_T  }{2} \leq \norma{\sum_{T \subseteq [n]}  a_T w_T  }{q},
\end{equation}  
where
\[
	c_q(\alpha,\beta) = \sqrt{  \frac{ \beta^{2-\frac{2}{q}} - \alpha^{2-\frac{2}{q}}  }{ \alpha \beta \left( \alpha^{-\frac{2}{q}}  - \beta^{-\frac{2}{q}}  \right) }  }.
\]
One can easily check that \eqref{hipsym} is a special case of \eqref{hipbiased}, namely $\sqrt{q-1}=\lim_{\e \to 0} c_q(\frac12-\e,\frac12+\e)$. Moreover, it is easy to see that $c_q(\alpha,\beta)\in [0,1]$. 

In \cite{FKN} the authors proved the following theorem, which is now called the FKN Theorem. Suppose $\alpha=\beta=\frac12$ and we have a Boolean function $f$ whose Fourier spectrum is concentrated on the first two levels, say $\sum_{|T|>1} a_T^2 < \e^2$. Then $f$ is $C\e$-close in the $L_2$ norm to the constant function or to one of the functions $\pm x_i$. The authors gave two proofs of this theorem. One of them contained an omission which was fixed by G. Kindler and S. Safra in their unpublished paper, \cite{KS}, see also \cite{K}.  In \cite{JOW} the authors gave a proof of the following version of the FKN Theorem. 
\begin{thm}[\cite{JOW}, Theorem 5.3 and Theorem 5.8]\label{thm1}
Let $f=\sum_T a_T w_T$ be the Walsh-Fourier expansion of a function $f:\{-1, 1 \}^n \to \{-1,1 \}$ and let $\rho = \Big(\sum_{|T| >1}a_T^2 \Big)^{1/2} $. Then there exists $B \subseteq [n]$ with $|B|\leq 1$ such that
$
	\sum_{|T| \leq 1, T \ne B} a_T^2 \leq C \rho^4 \ln(2/ \rho)
$  
and $|a_B|^2 \geq 1-\rho^2-C \rho^4 \ln(2/ \rho)$, where $C$ is a universal constant.

Moreover, in the non-symmetric case,  $f:\cube \to \{-1,1 \}$, there exists $k \in [n]$ such that $\norma{f-(a_\emptyset+a_{\{k \}}w_{\{ k\}})}{}\leq 8 \sqrt{\rho}$.
\end{thm}
This theorem is sharp, up to the universal constant $C$. In the proof the inequality \eqref{hipsym} has been used. However, in the non-symmetric case one can ask for a better bound involving bias parameter $\alpha$. In this note we use inequality \eqref{hipbiased} to prove such an extension of the FKN Theorem. Namely, we have
\begin{thm}\label{thm2}
Let $f=\sum_T a_T w_T$ be the Walsh-Fourier expansion of a function $f:\cube \to \{-1,1 \}$ and let $\rho = \Big(\sum_{|T| >1}a_T^2 \Big)^{1/2} $. Then there exists $k \in [n]$ and a universal constant $c_0>0$ such that for $\rho\ln(e/\rho)<c_0\alpha$ we have
%\[
%	\sum_{|S|\leq 1, \ S \ne \emptyset,\{k  \}}  a_S^2 \leq C \frac{\rho^4 \ln^3(1/ \rho)}{\alpha^3 \ln^3(1/ \alpha) }
%\]
%and 
\[
	\norma{f-(a_\emptyset+a_{\{k \}}w_{\{ k\}})}{}\leq 2\rho.
\]	
\end{thm}
Our proof of Theorem \ref{thm2}, which is given in the Section \ref{sec1}, is a straightforward application of the ideas used in the proof of Theorem 5.3 in \cite{JOW}.

In the Section \ref{sec2} we consider the case $\gamma=1$ and we deal with the problem concerning $[-1,1]$-valued functions defined on the cube $\kostka$ with uniform product probability measure. A function $f:\kostka \to \mb{R}$ is called \emph{affine} if $f(x)=a_0+\sum_{i=1}^n a_i x_i$, where $a_0,a_1,\ldots,a_n \in \mb{R}$ and $x=(x_1,\ldots,x_n)$. We will denote the set of all affine functions by $\aff$. Moreover, let $\affb \subset \aff$ stands for the set of all affine $[-1,1]$-valued functions. Note that $f \in \affb$ if and only if $\sum_{i=0}^n |a_i|\leq 1$. The function $f(x)=x_i$ will be denoted by $r_i$, $i=1,\ldots,n$. Let us also notice that if $f$ is $[-1,1]$-valued then $|a_T| = |\mb{E}w_T f| \leq \mb{E}|w_T f| \leq 1$.

In \cite{JOW} the authors gave the following example. Take $g:\kostka \to \mb{R}$ given by $g(x)=s^{-1}n^{-1/2} \sum_{i=1}^n x_i$. Note that $g \in \aff$. Define $\phi(x)=-\1_{(-\infty,-1)}(x)+x\1_{[-1,1]}(x)+\1_{(1,\infty)}(x)$ and take $f=\phi \circ g$. Clearly, $f$ is $[-1,1]$-valued but may not be affine. The authors proved that $\lim_{n \to \infty} \dist_{L^2}(f,\aff) = O(e^{-s^2/4})$ and $\lim_{n \to \infty} \dist_{L^2}(f,\affb) = \Theta(s^{-1})$.

Here we prove that this is the worst case. Namely, we have the following theorem.

\begin{thm}\label{thm3}
Let us take $f:\kostka \to [-1,1]$ and define $\rho = \dist_{L^2}(f,\aff)$. Then $\dist_{L_2}(f,\affb) \leq \frac{18}{\sqrt{\ln(1/\rho)}}$.
\end{thm}

In this paper we use the $\{-1,1 \}$-valued function $\sgn(x)=-\mb{I}_{(-\infty,0)}(x)+\mb{I}_{[0,\infty)}(x)$. By $C$ we denote a universal constant that may vary from one line to another.

\section{Proof of Theorem \ref{thm2}}\label{sec1}

Here we give a proof of Theorem \ref{thm2}. 

\begin{proof}
Let $k$ be given by Theorem \ref{thm1}. Let $h=f-(a_{\emptyset}+a_{\{k\}}x_k)$ and $\tilde{h}=f-\sgn(a_{\emptyset}+a_{\{k\}}x_k)$. Moreover, let $d=\norma{h}{}$. Note that for every $u \in \mb{R}$ and $\e \in \{-1,1\}$ we have $|u-\sgn(u)| \leq |u-\e|$. Therefore,
\[
	|\e-\sgn(u)| \leq |\e-u|+|u-\sgn(u)| \leq 2|u-\e|.
\]
It follows that $|\tilde{h}|\leq 2|h|$. Thus, using the fact that $\tilde{h}$ is $\{-2,0,2 \}$-valued, we have
\[
	\mb{P}(\tilde{h} \ne 0)= \frac14 \| \tilde{h} \|^2 \leq \norma{h}{}^2 = d^2.
\]
Let us consider the expansion $\tilde{h}=\sum_{T}\tilde{a}_T w_T$. Clearly, $\tilde{a}_T=a_T$ for $T \ne \emptyset,\{ k\}$. Using \eqref{hipbiased} we obtain
\begin{align*}
	4d^{4/q} & \geq 4 \mb{P}(\tilde{h}\ne 0)^{2/q} = \| \tilde{h} \|_q^2 = \norma{\sum_T \tilde{a}_T w_T}{q}^2 \geq \norma{\sum_T c_q(\alpha,\beta)^{|T|} \tilde{a}_T w_T}{2}^2 \\
	& = \sum_{T} c_q(\alpha,\beta)^{2|T|} \tilde{a}_T^2 \geq c_q(\alpha,\beta)^{2} \sum_{|T|\leq 1}  \tilde{a}_T^2,
\end{align*}
where $q \in [1,2]$. We arrive at
\[
	\sum_{|T|\leq 1, \ T \ne \emptyset,\{k  \}}  \tilde{a}_T^2 \leq \sum_{|T|\leq 1}  \tilde{a}_T^2 \leq \frac{4d^{4/q}}{c_q(\alpha,\beta)^{2} }  = 4d^{4/q} \alpha \beta \cdot \frac{\alpha^{-\frac{2}{q}} - \beta^{-\frac{2}{q}}  }{\beta^{2-\frac{2}{q}}-\alpha^{2-\frac{2}{q}}}  .
\]
Let $p=\frac{\beta}{\alpha}\in(1,\infty)$ and $x=p^{2/q}\in[p,p^2]$. Then we have
\[
	4d^{4/q} \alpha \beta \frac{\alpha^{-\frac{2}{q}} - \beta^{-\frac{2}{q}}  }{\beta^{2-\frac{2}{q}}-\alpha^{2-\frac{2}{q}}} = 4d^{\frac{2\ln x}{\ln p}} \alpha \beta \frac{1-\frac{1}{x}}{\frac{\beta^2}{x}-\alpha^2 } = 4d^{\frac{2\ln x}{\ln p}} p \frac{x-1}{p^2-x} \leq 4d^{\frac{2\ln x}{\ln p}} p \frac{p^2-1}{p^2-x}.
\] 
Without loss of generality, taking sufficiently small $c_0>0$, we can assume that $\rho \leq \frac{1}{64 \cdot 9}$. From Theorem \ref{thm1} we obtain $d \leq 8\sqrt{\rho} \leq \frac13 \leq \frac{1}{e}$. Therefore, $\frac{1}{\ln(1/d)} \in [0,1]$ and we can take $x=p^{2-\frac{1}{\ln(1/d)}}$. Then 
\[
	4d^{\frac{2\ln x}{\ln p}} p \frac{p^2-1}{p^2-x} = 4e^2 d^4 \frac{p^2-1}{p} \cdot \frac{1}{1-e^{-\frac{\ln p}{\ln(1/d)}}}.
\]
For $c_0 \in (0,1)$ the condition $\rho\ln(e/\rho) \leq c_0 \alpha$ clearly implies $\rho \leq \alpha$. Since $d \leq 8 \sqrt{\rho} \leq \min\{ \frac13, 8\sqrt{\alpha}\}$, then one can easily check that there exists a constant $c_1$ such that $0 \leq \frac{\ln p}{\ln(1/d)} \leq c_1$. There exists a constant $c_2$ such that for all $s \in [0,c_1]$ we have $\frac{1}{1-e^{-s}}\leq \frac{c_2}{s}$. Thus,
\[
	 4e^2 d^4 \frac{p^2-1}{p} \cdot \frac{1}{1-e^{-\frac{\ln p}{\ln(1/d)}}} \leq 8c_0e^2 d^4 \frac{p-1}{\ln p} \ln(1/d) \leq C \frac{d^4 \ln(1/d)}{\alpha \ln(1/ \alpha)} .
\] 
We arrive at
\[
	\sum_{|T|\leq 1, \ T \ne \emptyset,\{k  \}}  \tilde{a}_T^2 \leq C \frac{d^4 \ln(1/d)}{\alpha \ln(1/ \alpha)}  \leq \frac{C}{\alpha} \cdot d^4 \ln(1/d).
\]
We have 
\begin{align*}
	d^2 & = \norma{f-(a_{\emptyset}+a_{\{k\}}x_k)}{}^2 = \norma{f}{}^2 + \norma{a_{\emptyset}+a_{\{k\}}x_k}{}^2 - 2 \skal{f}{a_{\emptyset}+a_{\{k\}}x_k} \\
	& = 1 + a_{\emptyset}^2+a_{\{k\}}^2  - 2(a_{\emptyset}^2+a_{\{k\}}^2   ) = 1-a_{\emptyset}^2+a_{\{k\}}^2   .  
\end{align*}
Thus, $a_{\emptyset}^2+  a_{\{k\}}^2  = 1-d^2$ and we can write

\[
	\sum_{|T|\leq 1} a_T^2   \leq a_\emptyset^2 +a_{\{k \}}^2  + \frac{C}{\alpha} \cdot d^4 \ln(1/d) = 1-d^2+ \frac{C}{\alpha} \cdot d^4 \ln(1/d).
\]
It follows that
\[
	\rho^2 = \sum_{|T|>1} a_T^2  = 1 - \sum_{|T|\leq 1} a_T^2  \geq d^2 -\frac{C}{\alpha} \cdot d^4 \ln(1/d).
\]
Since from Theorem \ref{thm1} we know that $\rho \leq d \leq 8\sqrt{\rho}$, we obtain
\[
	\frac{C}{\alpha} \cdot d^2 \ln(1/d) \leq \frac{64C}{\alpha} \rho \ln(1/\rho) \leq 64Cc_0 \leq \frac{3}{4},
\]
assuming that $\rho \ln(e/\rho)\leq c_0\alpha$ and $c_0>0$ is sufficiently small. 
Therefore,
\[
	\rho^2 \geq  d^2 -\frac{C}{\alpha} \cdot d^4 \ln(1/d) \geq d^2- \frac{3}{4}d^2 = \frac14 d^2.
\]
It follows that $d \leq 2\rho$.
\end{proof}

\begin{rem*}
The condition $\rho \ln(e/ \rho) \leq c_0 \alpha$ cannot be replaced by $\rho \leq \alpha$. Indeed, if we take $f:\{-\gamma,\gamma^{-1} \}^2 \to \{-1,1 \}$ given by 
\[
	f(x_1,x_2) = 2(\beta-\sqrt{\beta \alpha}x_1)(\beta-\sqrt{\beta \alpha}x_2)-1,
\] 
see the remark after the proof of Theorem 5.8 in \cite{JOW}, then we obtain $\rho=2\alpha \beta$ and $d=2\beta^{3/2}\alpha^{1/2}$. Thus $d = \sqrt{2\rho} \beta \geq \sqrt{\rho/2}$.  
\end{rem*}

\section{Proof of Theorem \ref{thm3}}\label{sec2}

We need the following lemma due to P. Hitczenko and S. Kwapień.

\begin{lem}
(\cite{HK}, Theorem 1 and \cite{O}, Theorem 1) Let $a_1 \geq a_2 \geq \ldots \geq a_n \geq 0$ and let us take $S:\kostka \to \mb{R}$ given by $S=\sum_{i=1}^n a_i r_i$. Then for $t \geq 1$ we have 
\begin{equation}\label{eq1}
\p{|S| \geq \norma{S}{}} > \frac{1}{10}
\end{equation} 
and
\begin{equation}\label{eq2}
\norma{S}{t} \geq \frac14  \sqrt{t} \Big(\sum_{i>t} a_i^2 \Big)^{1/2}.
\end{equation}
\end{lem}

We give a proof of Theorem \ref{thm3}.

\begin{proof}[Proof of Theorem \ref{thm3} ]
\emph{Step 1.} If $f=\sum_T a_T w_T$ then $\dist_{L_2}(f,\aff)=\norma{f-S }{}$, where $S=\sum_{|T|\leq 1} a_T w_T$. For every $u \in [-1,1]$ we have $|x-u| \geq |x-\phi(x)|$ for all $x \in \mb{R}$. Taking $x=S$ and $u=f$ we obtain $\mb{E}(|S|-1)_+^2 = \norma{S-\phi(S)}{}^2 \leq \norma{S-f}{}^2 \leq \rho^2$. For all $g \in \affb$ we have
\[
	\norma{g-f}{} \leq \norma{g-S}{}+\norma{S-f}{} \leq \norma{g-S}{} + \rho.
\]
Therefore, 
\begin{equation}\label{eq3}
	\dist_{L_2}(f,\affb) \leq \dist_{L_2}(S,\affb)+\rho.
\end{equation} 
It suffices to prove that $\mb{E}(|S|-1)_+^2 \leq \rho^2$ implies an appropriate bound on $\dist_{L_2}(S,\affb)$,  whenever $S=a_0+\sum_{i=1}^n a_ir_i$, where $a_0,a_1,\ldots,a_n \in \mb{R}$.

\emph{Step 2.} Suppose that for all $n \geq 1$ we can prove that $\mb{E}(|S|-1)_+^2 \leq \rho^2$ implies $\dist_{L_2}(S,\affb) \leq M$ for some $M>0$, assuming that $a_0=0$. Then we can deal with the case $a_0 \ne 0$ as follows. Let us take $\tilde{S}:\{-1,1\} \times \kostka \to \mb{R}$ given by $\tilde{S}=a_0x_0+\sum_{i=1}^n a_i x_i$. Clearly, $\mb{E}(|\tilde{S}|-1)_+^2=\mb{E}(|S|-1)_+^2 \leq \rho^2 $. We can find a $[-1,1]$-valued function $\tilde{S}_0=b_0x_0+\sum_{i=1}^n b_i x_i$ such that $\norma{\tilde{S}-\tilde{S_0}}{} \leq M$. Take $S_0=b_0+\sum_{i=1}^n b_i x_i$. Now it suffices to observe that the function $S_0$ is $[-1,1]$-valued and to notice that $\norma{\tilde{S}-\tilde{S_0}}{}=\norma{S-S_0}{}$.     

\emph{Step 3.} Take $S=\sum_{i=1}^n a_i r_i$. Without loss of generality we can assume that $1 \geq a_1 \geq a_2 \geq \ldots \geq a_n \geq 0$. Let $\tau=\max\{ t \geq 1: \sum_{i=1}^t a_i \leq 1 \}$. Clearly, $\tau \geq 1$. If $f$ is already in $\affb$ then there is nothing to prove. Therefore we can assume that $\tau<n$. We can also assume that $\rho \leq 1/3$, since otherwise we have 
\[
	\dist_{L_2}(f,\affb) \leq \dist_{L_2}(f,0) = \norma{f}{} \leq 1 \leq \frac{18}{\sqrt{\ln(1/ \rho)}}.
\]
Let $A=\{ |S| \geq \frac12 \norma{S}{t} \}$. For $t \geq 1$ we have
\[
	\mb{E}|S|^t  = \mb{E}|S|^t \1_{A} + \mb{E}|S|^t \1_{A^c} 
	 \leq \sqrt{\mb{E}|S|^{2t}} \sqrt{\mb{P}(A)} + \frac{1}{2^t} \mb{E}|S|^t.
\]
Since by the Khinchine inequality we have $(\mb{E}|S|^{2t})^{1/2t} \leq \sqrt{\frac{2t-1}{t-1}}(\mb{E}|S|^{t})^{1/t}$, we obtain
\[
	\mb{P}\left(|S| \geq \frac12 \norma{S}{t}\right) \geq \left(1-\frac{1}{2^t} \right)^2\frac{(\mb{E}|S|^t)^2}{\mb{E}|S|^{2t}} \geq \frac14 \frac{(\mb{E}|S|^t)^2}{\mb{E}|S|^{2t}} \geq \frac14 \left( \frac{t-1}{2t-1} \right)^t. 
\]
By the Chebyshev inequality we obtain
\begin{equation}\label{eq4}
	\p{|S| \geq 1+\e} \leq \frac{\mb{E}(|S|-1)_+^2}{\e^2} \leq \frac{\rho^2}{\e^2},
\end{equation}
for all $\e>0$.
Let $t \geq 1$ and assume that $\norma{S}{t}>2$. Take $\e =\frac12 \norma{S}{t}-1>0$. We obtain
\[
	\frac14 \left( \frac{t-1}{2t-1} \right)^t \leq \p{|S| \geq \frac12 \norma{S}{t}} \leq \frac{\rho^2}{\left( \frac12 \norma{S}{t} -1   \right)^2}.
\]
It follows that 
\[
	\norma{S}{t} \leq 2+4 \rho \left( \frac{2t-1}{t-1} \right)^{t/2} 
\] 
which is also true in the case $\norma{S}{t} \leq 2$. From inequality \eqref{eq2} we obtain  
\begin{equation}\label{st}
\frac14 \sqrt{t} \Big(\sum_{i>t} a_i^2 \Big)^{1/2}	\leq  \norma{S}{t} \leq 2+4 \rho \left( \frac{2t-1}{t-1} \right)^{t/2}.
\end{equation}
We consider the case $\tau \geq \frac{2}{\ln 3} \ln(1/\rho) \geq 1$. 
Let us now take $t=\frac{2}{\ln 3} \ln(1/\rho)\geq 2>1$ and define $S_1=\sum_{i \leq \frac{2}{\ln 3} \ln(1/\rho)} a_i r_i$. Notice that we have $\sum_{i \leq \frac{2}{\ln 3} \ln(1/\rho)} a_i \leq \sum_{i \leq \tau} a_i \leq 1$ . Thus, $S_1 \in \affb$. Moreover, since $t\geq 2$, we have 
$
	 \rho \left( \frac{2t-1}{t-1} \right)^{t/2} \leq \rho 3^{t/2} = 1
$ 
and therefore by \eqref{st} we have
\[
	\dist_{L_2}(S,\affb)  \leq \norma{S-S_1}{} = \left(\sum_{i> \frac{2}{\ln 3} \ln(1/ \rho)} a_i^2 \right)^{1/2} \leq \frac{24 }{\sqrt{ \frac{2}{\ln 3} \ln(1/ \rho)}}.
\]
In this case \eqref{eq3} yields 
\[
	\dist_{L_2}(f,\affb) \leq \frac{24 }{\sqrt{ \frac{2}{\ln 3} \ln(1/ \rho)}}+ \rho \leq \frac{18 }{\sqrt{\ln(1/ \rho)}}.
\]

\emph{Step 4.} We are to deal with the case $\tau< \frac{2}{\ln 3}\ln(1/ \rho)$. Let $S_2=\sum_{i \geq \tau+2} a_i r_i$. From inequality \eqref{eq1} we have 
\[
	\p{|S| \geq \sum_{i \leq \tau+1} a_i + \norma{S_2}{}  }  \geq \frac{1}{2^{\tau+1}} \p{|S_2| \geq \norma{S_2}{}} \geq \frac{1}{2^{\tau+1}} \cdot \frac{1}{10} \geq \frac{1}{20} \rho^{\frac{2\ln 2}{\ln 3}}.
\]
Note that $\sum_{i \leq \tau+1} a_i > 1$. Therefore, from inequality \eqref{eq4} we obtain  
\[
	\p{|S| \geq \sum_{i \leq \tau+1} a_i + \norma{S_2}{} }	 \leq \frac{\rho^2}{ \left(  \sum_{i \leq \tau+1} a_i + \norma{S_2}{} -1  \right)^2  }.
\]
It follows that
\[
	\sum_{i \leq \tau+1} a_i + \norma{S_2}{} -1  \leq  \sqrt{20} \rho^{1-\frac{\ln 2}{\ln 3}}. 
\]
Take $S_1 = \sum_{i=1}^\tau a_i r_i + (1-(a_1+\ldots+a_\tau))r_{\tau+1}$. Clearly, $S_1 \in \affb$. Moreover,
\begin{align*}
	\norma{S-S_1}{} & = \left((1-(a_1+\ldots+a_\tau)-a_{\tau+1})^2 + \norma{S_2}{}^2 \right)^{1/2} \\
	& \leq |a_1+\ldots+a_\tau+a_{\tau+1}-1| + \norma{S_2}{} \leq \sqrt{20} \rho^{1-\frac{\ln 2}{\ln 3}}.
\end{align*}
Therefore, from \eqref{eq3} we have
\[
	\dist_{L_2}(f,\affb) \leq \sqrt{20} \rho^{1-\frac{\ln 2}{\ln 3}} + \rho \leq \frac{18 }{\sqrt{\ln(1/ \rho)}}.
\]
\end{proof}

\begin{rem*}
If we perform our calculation with $\ln(2.03)$ instead of $\ln 3$ we will obtain the theorem with a constant $14,5$ instead of $18$. 
\end{rem*}

\section*{Acknowledgements}

I would like to thank Prof. Krzysztof Oleszkiewicz for his useful comments.

\noindent Piotr Nayar$^\star$, \texttt{nayar@mimuw.edu.pl}

\vspace{2em}

\noindent $^\star$Institute of Mathematics, University of Warsaw, \\
\noindent Banacha 2, \\
\noindent 02-097 Warszawa, \\
Poland.

\end{document}